\documentclass[a4paper]{amsart}
\usepackage{graphicx}
\usepackage{amsmath}
\usepackage{amssymb}
\usepackage{oldgerm}
\usepackage[all]{xy}

\newcommand{\Hom}{\operatorname{Hom}\nolimits}

\renewcommand{\mod}{\operatorname{mod}\nolimits}

\newcommand{\Ext}{\operatorname{Ext}\nolimits}

\newcommand{\La}{\Lambda}

\newcommand{\cx}{\operatorname{cx}\nolimits}

\newcommand{\T}{\operatorname{\mathcal{T}}\nolimits}
\newcommand{\C}{\operatorname{\mathcal{C}}\nolimits}
\newcommand{\D}{\operatorname{\mathcal{D}}\nolimits}

\newcommand{\End}{\operatorname{End}\nolimits}
\newcommand{\CM}{\operatorname{CM}\nolimits}
\newcommand{\add}{\operatorname{add}\nolimits}
\newcommand{\thick}{\operatorname{thick}\nolimits}
\newcommand{\s}{\operatorname{\Sigma}\nolimits}
\newcommand{\Tr}{\operatorname{Tr}\nolimits}

\newtheorem{theorem}{Theorem}[section]
\newtheorem*{unnmbrdtheorem}{Theorem}

\newtheorem{corollary}[theorem]{Corollary}

\newtheorem{lemma}[theorem]{Lemma}
\newtheorem{proposition}[theorem]{Proposition}
\theoremstyle{definition}

\theoremstyle{definition}

\theoremstyle{definition}
\newtheorem*{questions}{Questions}
\theoremstyle{definition}

\theoremstyle{definition}
\newtheorem*{examples}{Examples}

\theoremstyle{definition}

\theoremstyle{remark}
\newtheorem*{remark}{Remark}
\theoremstyle{definition}

\theoremstyle{definition}

\DeclareMathOperator{\stabCM}{\underline{CM}}

\begin{document}

\title{Cluster tilting and complexity}
\author{Petter Andreas Bergh \& Steffen Oppermann}
\address{Institutt for matematiske fag \\
  NTNU \\ N-7491 Trondheim \\ Norway}
\email{bergh@math.ntnu.no} \email{Steffen.Oppermann@math.ntnu.no}

\thanks{Both authors were supported by NFR Storforsk grant no.\
167130}



\maketitle

\begin{abstract}
We study the notion of positive and negative complexity of pairs of
objects in cluster categories. The first main result shows that the
maximal complexity occurring is either one, two or infinite,
depending on the representation type of the underlying hereditary
algebra. In the the second result, we study the bounded derived
category of a cluster tilted algebra, and show that the maximal
complexity occurring is either zero or one whenever the algebra is
of finite or tame type.
\end{abstract}

\section{Introduction}

Cluster categories associated to finite dimensional hereditary
algebras were introduced in \cite{BMRRT}. These $2$-Calabi-Yau
triangulated categories arise as orbit categories of derived
categories, and provide a categorification of the combinatorics of
the cluster algebras introduced in \cite{FominZelevinsky} by Fomin
and Zelevinsky in the acyclic case. They also provide a
generalized framework for classical tilting theory, with the
cluster tilting objects and their endomorphism rings, the cluster
tilted algebras.

Given two objects in a triangulated category defined over a field,
their total cohomology is a $\mathbb{Z}$-graded vector space over
the ground field. It therefore makes sense to study the rate of
growth of the dimensions in both the ``negative" and the ``positive"
direction, thus leading to the notion of negative and positive
complexity. In this paper, we study the complexity of a cluster
category, and show that the maximal complexity occurring depends on
the representation type of the hereditary algebra we start with:

\begin{unnmbrdtheorem}
Let $H$ be a basic finite dimensional hereditary algebra over an
algebraically closed field, and let $\C_H$ be the corresponding
cluster category. Then
$$\sup \{ \cx^*_{\C_H}(X,Y) \mid X,Y \in \C_H \} = \left \{
\begin{array}{ll}
1 & \text{if $H$ has finite type,} \\
2 & \text{if $H$ has tame type,} \\
\infty & \text{if $H$ has wild type.}
\end{array}
\right.$$
\end{unnmbrdtheorem}

We also study the complexity of the derived category of a cluster
tilted algebra, and show that in this case, the maximal complexity
occurring depends on the representation type of the algebra:

\begin{unnmbrdtheorem}
If $\Lambda$ is a cluster tilted algebra of finite or tame
representation type, then
\[ \sup \{ \cx^*_{\D^b(\Lambda)}(X,Y) \mid X,Y \in \D^b(\Lambda) \} = \left \{
\begin{array}{ll}
0 & \text{if $\Lambda$ is hereditary,} \\
1 & \text{otherwise}
\end{array}
\right. \]
\end{unnmbrdtheorem}

We prove this by showing that a tame cluster tilted algebra has
finitely many indecomposable Cohen-Macaulay modules. Finally, we
look at some examples showing what can happen for wild cluster
titled algebras.

\section{Preliminaries}

Throughout this section, we fix a field $k$ and a triangulated
$\Hom$-finite $k$-category $\T$ with suspension functor $\s$. Thus
for all objects $X,Y,Z$ in $\T$, the set $\Hom_{\T}(X,Y)$ is a
finite dimensional $k$-vector space, and the composition
$$\Hom_{\T}(Y,Z) \times \Hom_{\T}(X,Y) \to \Hom_{\T}(X,Z)$$
is $k$-bilinear. Recall that a \emph{Serre functor} on $\T$ is a
triangle equivalence $\T \xrightarrow{S} \T$, together with
functorial isomorphisms
$$\Hom_{\T}(X,Y) \simeq D \Hom_{\T}(Y,SX)$$
of vector spaces for all objects $X,Y \in \T$, where $D =
\Hom_k(-,k)$. By \cite{BondalKapranov}, such a functor is unique if
it exists. For an integer $d \in \mathbb{Z}$, the category $\T$ is
said to be \emph{$d$-Calabi-Yau} if it admits a Serre functor which
is isomorphic as a triangle functor to $\s^d$.

A subcategory of $\T$ is \emph{thick} if it is a full triangulated
subcategory closed under direct summands. Now let $\C$ and $\D$ be
subcategories of $\T$. We denote by $\thick^1_{\T} ( \C )$ the full
subcategory of $\T$ consisting of all the direct summands of finite
direct sums of shifts of objects in $\C$. Furthermore, we denote by
$\C \ast \D$ the full subcategory of $\T$ consisting of objects $M$
such that there exists a distinguished triangle
$$C \to M \to D \to \s C$$
in $\T$, with $C \in \C$ and $D \in \D$. Now for each $n \ge 2$,
define inductively $\thick^n_{\T} ( \C )$ to be $\thick_{\T}^1 \left
( \thick^{n-1}_{\T} ( \C ) \ast \thick^1_{\T} ( \C ) \right )$, and
denote $\bigcup_{n=1}^{\infty} \thick^n_{\T} ( \C )$ by $\thick_{\T}
( \C )$. This is the smallest thick subcategory of $\T$ containing
$\C$.

Given two objects $X$ and $Y$ of $\T$, we define the \emph{positive
complexity} of the ordered pair $(X,Y)$ as
$$\cx_{\T}^+ (X,Y) \stackrel{\text{def}}{=} \inf \{ t \in \mathbb{N}
\cup \{ 0 \} \mid \exists a \in \mathbb{R}: \dim \Hom_{\T}(X, \s^nY)
\le an^{t-1} \text{ for } n \gg 0 \}.$$ Similarly, we define the
\emph{negative complexity} as
$$\cx_{\T}^- (X,Y) \stackrel{\text{def}}{=} \inf \{ t \in \mathbb{N}
\cup \{ 0 \} \mid \exists a \in \mathbb{R}: \dim \Hom_{\T}(X,
\s^{-n}Y) \le an^{t-1} \text{ for } n \gg 0 \}.$$ Whenever we write
$\cx_{\T}^*(X,Y)$ and make a statement, it is to be understood that
the statement holds for both the positive and the negative
complexity. By definition, the positive complexity is zero if and
only if $\Hom_{\T}(X, \s^nY)=0$ for large $n$, whereas the negative
complexity is zero if and only if $\Hom_{\T}(X, \s^nY)=0$ for small
$n$. Moreover, given integers $a,b \in \mathbb{Z}$, there is an
equality $\cx_{\T}^*(X,Y) = \cx_{\T}^*( \s^aX, \s^bY)$. Note also
that if $\T$ is $d$-Calabi-Yau for some $d$, then $\cx_{\T}^+(X,Y) =
\cx_{\T}^-(Y,X)$; in particular the equality $\cx_{\T}^+(X,X) =
\cx_{\T}^-(X,X)$ holds in this case.

The following elementary lemma shows that complexity in some sense
behaves nicely on thick subcategories.

\begin{lemma}\label{thickcx}
Let $X$ and $Y$ be objects of $\T$. Then $\cx_{\T}^*(X',Y) \le
\cx_{\T}^*(X,Y)$ for all objects $X' \in \thick_{\T} ( X )$, and
$\cx_{\T}^*(X,Y') \le \cx_{\T}^*(X,Y)$ for all objects $Y' \in
\thick_{\T} ( Y )$. In particular, the inequality
$\cx_{\T}^*(X',X'') \le \cx_{\T}^*(X,X)$ holds for all objects
$X',X'' \in \thick_{\T} ( X )$.
\end{lemma}

\begin{proof}
We prove only the first inequality, by induction on the number $n$
such that $X'$ belongs to $\thick^n_{\T} ( X )$. If $\cx_{\T}^*(X,Y)
= \infty$, then the inequality obviously holds. Hence we may assume
that $\cx_{\T}^*(X,Y)$ is finite, say $\cx_{\T}^*(X,Y)=c$. If $n=1$,
then $X'$ is a direct summand of finite direct sums of shifts of
$X$, hence the inequality holds in this case. Next, suppose $n>1$
and that $X'$ belongs to $\thick^{n-1}_{\T} (X) \ast \thick^1_{\T}
(X)$. Then there exists a triangle
$$X_1 \to X' \to X_2 \to \s X_1$$
in which $X_1 \in \thick^{n-1}_{\T} (X)$ and $X_2 \in \thick^1_{\T}
(X)$. This triangle induces an exact sequence
$$\Hom_{\T}(X_2, \s^nY) \to \Hom_{\T}(X', \s^nY) \to \Hom_{\T}(X_1,
\s^nY)$$ of vector spaces for every $n \in \mathbb{Z}$. By
induction, both $\cx_{\T}^*(X_1,Y)$ and $\cx_{\T}^*(X_2,Y)$ are at
most $c$. Therefore, there exist real numbers $a_1$ and $a_2$ such
that
\begin{eqnarray*}
\dim \Hom_{\T}(X_1, \s^nY) & \le & a_1|n|^{c-1} \\
\dim \Hom_{\T}(X_2, \s^nY) & \le & a_2|n|^{c-1}
\end{eqnarray*}
for $|n| \gg 0$. This gives
\begin{eqnarray*}
\dim \Hom_{\T}(X', \s^nY) & \le & \dim \Hom_{\T}(X_1, \s^nY) + \dim
\Hom_{\T}(X_2, \s^nY) \\
& \le & (a_1+a_2)|n|^{c-1}
\end{eqnarray*}
for $|n| \gg 0$, showing that $\cx_{\T}^*(X',Y)$ is at most $c$. The
result now follows from the definition of $\thick^{n}_{\T} (X)$.
\end{proof}

The aim of this paper is to determine the maximal complexity
occurring in certain triangulated categories, via maximal
orthogonal subcategories. Recall that a subcategory $\C$ of $\T$
is \emph{contravariantly finite} in $\T$ if every object in $\T$
admits a right $\C$-approximation. Thus, for every object $X \in
\T$ there exists a morphism $C \to X$ with $C \in \C$, such that
every morphism $C' \to X$ with $C' \in \C$ factors through $C$.
The following lemma provides a criterion under which a
contravariantly finite subcategory $\C$ generates $\T$ (see
\cite{Iyama1} and \cite[Section 5.5]{KellerReiten}). Consequently,
we see from Lemma \ref{thickcx} that the maximal complexity of
$\T$ equals that of $\C$.

\begin{lemma}\label{generating}
Let $\C$ be a contravariantly finite subcategory of $\T$, and
suppose there exists an integer $n \ge 1$ such that the following
are equivalent for every object $X \in \T$:
\begin{enumerate}
\item $X \in \C$, \item $\Hom_{\T}(C, \s^iX)=0$ for $1 \le i \le
n$ and all $C \in \C$.
\end{enumerate}
Then $\thick_{\T}^{n+1}( \C ) = \T$.
\end{lemma}

\begin{proof}
Choose $n$ triangles
$$\xymatrix@R=0.7pc{
K_1 \ar[r] & C_0 \ar[r]^{f_0} & X \ar[r] & \s K_1 \\
\vdots & \vdots & \vdots & \vdots \\
K_{n-1} \ar[r] & C_{n-2} \ar[r]^{f_{n-2}} & K_{n-2} \ar[r] & \s K_{n-1} \\
K_n \ar[r] & C_{n-1} \ar[r]^{f_{n-1}} & K_{n-1} \ar[r] & \s K_n }$$
in which the morphisms $f_i$ are right $\C$-approximations. Let $C$
be any object in $\C$. The triangles induce exact sequences
$$\cdots \to \Hom_{\T}(C, \s^j C_i) \xrightarrow{( \s^j f_i )_*}
\Hom_{\T}(C, \s^j K_i) \to \Hom_{\T}(C, \s^{j+1} K_{i+1}) \to
\cdots$$ for $0 \le i \le n-1$ (where we have denoted $X$ by $K_0$).
An induction argument shows that $\Hom_{\T}(C, \s^i K_n)$ vanishes
for $1 \le i \le n$, hence $K_n$ belongs to $\C$. Then another
induction argument shows that $X$ belongs to $\thick_{\T}^{n+1} ( \C
)$.
\end{proof}

\begin{corollary}\label{maxcx}
Given the assumptions from the previous lemma, the equality
$$\sup \{ \cx^*_{\T}(X,Y) \mid X,Y \in \T \} = \sup \{ \cx^*_{\T}(C,C') \mid C,C' \in \C
\}$$ holds.
\end{corollary}

In the next section, we apply the above results to Calabi-Yau
triangulated categories admitting subcategories with the
properties displayed in the assumption of Lemma \ref{generating}.
Recall therefore that, if $\T$ is $d$-Calabi-Yau for some $d \ge
2$, then a \emph{cluster tilting subcategory} of $\T$ is a
contravariantly finite subcategory $\C$ such that the following
are equivalent for any object $X \in \T$:
\begin{enumerate}
\item $X \in \C$, \item $\Hom_{\T}(C, \s^iX)=0$ for $1 \le i \le
d-1$ and all $C \in \C$.
\end{enumerate}
Since $\T$ is $d$-Calabi-Yau, property (2) is equivalent to
\begin{enumerate}
\item[(3)] $\Hom_{\T}( X, \s^iC)=0$ for $1 \le i \le
d-1$ and all $C \in \C$.
\end{enumerate}
An object $T \in \T$ is a \emph{cluster tilting object} of $\T$ if
$\add T$ is a cluster tilting subcategory.

Note that it follows directly from Corollary \ref{maxcx} that if
$\C_1$ and $\C_2$ are cluster tilting subcategories of $\T$, then
\begin{eqnarray*}
\sup \{ \cx^*_{\T}(X,Y) \mid X,Y \in \T \} & = & \sup \{
\cx^*_{\T}(C_1,C_1') \mid C_1,C_1' \in \C_1 \} \\
& = & \sup \{ \cx^*_{\T}(C_2,C_2') \mid C_2,C_2' \in \C_2 \}.
\end{eqnarray*}
Therefore, in order to determine the maximal complexity of a
Calabi-Yau triangulated category, any cluster tilting subcategory
will do.

\section{Cluster categories}

Cluster categories associated to finite dimensional hereditary
algebras were introduced in \cite{BMRRT} (and for hereditary
algebras of Dynkin type $A_n$ in \cite{CalderoChapotonSchiffler}).
Let $k$ be an algebraically closed field and $H$ a basic finite
dimensional hereditary $k$-algebra. Let $\D^b( H)$ be the bounded
derived category of finitely generated left $H$-modules; this
category is triangulated, its suspension functor $\s$ is just the
shift of a complex. Finally, denote by $\tau$ the Auslander-Reiten
translate in $\D^b(H)$; this functor is induced by the usual
Auslander-Reiten translate $D \Tr$ on the non-projective
indecomposable $H$-modules. It was shown in \cite{Keller} that the
orbit category $\D^b(H) / \tau^{-1} \s$ is triangulated, with
suspension functor induced by $\s$. This is the \emph{cluster
category} $\C_H$ associated to $H$. Its objects coincide with the
objects in $\D^b(H)$, and the functors $\s$ and $\tau$ are equal.
Given objects $X$ and $Y$ of $\C_H$, the morphism space
$\Hom_{\C_H}(X,Y)$ is given by
$$\Hom_{\C_H}(X,Y) \stackrel{\text{def}}{=} \bigoplus_{i \in
\mathbb{Z}} \Hom_{\D^b(H)}( \tau^{-i} \s^i X,Y ),$$ which is finite
dimensional since $H$ is hereditary. We shall denote the suspension
functor of $\C_H$ by $\s$ as well. Moreover, given an $H$-module
$M$, we shall also denote its image in $\C_H$ by $M$. By
\cite[Proposition 1.7(b)]{BMRRT} the cluster category $\C_H$ is
$2$-Calabi-Yau, that is, there is an isomorphism
$$D \Hom_{\C_H}(X, \s Y) \simeq \Hom_{\C_H}(Y, \s X)$$
of vector spaces for all objects $X$ and $Y$ in $\C_H$.

In order to prove the main result, we need a result on the rate of
growth of the sequence $\{ \dim \tau^{-n}H \}_{n=1}^{\infty}$ for
a hereditary algebra $H$. Recall first that the representation
type of a finite dimensional algebra (over an algebraically closed
field) is either \emph{finite, tame} or \emph{wild}. An algebra is
of finite representation type if there are only finitely many
non-isomorphic indecomposable modules. Furthermore, an algebra is
of tame representation type if there exist infinitely many
non-isomorphic indecomposable modules, but they all belong to
one-parameter families, and in each dimension there are finitely
many such families. Finally, an algebra is of wild representation
type if it is not of finite or tame type. In the latter case, the
representation theory of the algebra is at least as complicated as
the classification of finite dimensional vector spaces together
with two non-commuting endomorphisms.

\begin{proposition}\label{taugrowth}
Let $H$ be a finite dimensional hereditary algebra of infinite
representation type over an algebraically closed field. Define
$$\gamma \left ( \tau^{-1}_H \right ) := \inf \{ t \in \mathbb{N}
\cup \{ 0 \} \mid \exists a \in \mathbb{R}: \dim \tau^{-n}H \le
an^{t-1} \text{ for } n \gg 0 \}.$$ Then the following hold:
\begin{enumerate}
\item $\gamma \left ( \tau^{-1}_H \right ) =2$ if \emph{(}and only
if\emph{)} $H$ is tame. \item $\gamma \left ( \tau^{-1}_H \right )
= \infty$ if \emph{(}and only if\emph{)} $H$ is wild.
\end{enumerate}
\end{proposition}

\begin{proof}
($1$) Suppose $H$ is tame. We may assume that $H$ is the path
algebra of one of the Euclidean quivers $\widetilde{A}_n,
\widetilde{D}_n, \widetilde{E}_6, \widetilde{E}_7,
\widetilde{E}_8$. We prove this case by using the theory of
quadratic forms, and refer to \cite[Chapter 1]{Ringel1} for
unexplained notation and terminology. Let $n$ be the number of
vertices of the underlying Euclidean quiver.

Let $r \in \mathbb{Z}^n$ be a minimal positive radical vector. We
claim that the set
$$\{ x \in \mathbb{Z}^n \mid x \text{ root}, 0 \le x, r \nleq x
\}$$ is finite. To see this, note that there exists an integer $1
\le i \le n$ such that the $i$th entry in $r$ is $1$ (see
\cite[table on page 8]{Ringel1}). This implies that $\{ r  \} \cup
\{ e_j \mid 1 \le j \le n,j \neq i \}$ is a $\mathbb{Z}$-basis for
$\mathbb{Z}^n$, where $e_j$ denotes the $j$th unit vector. With
respect to this basis, the quadratic form is given by a matrix
$$\left ( \begin{array}{cc}
A & 0 \\
0 & 0
\end{array} \right )$$
in which $A$ is the matrix of the quadratic form of the
corresponding Dynkin quiver. Thus all the roots are of the form $y
+ \alpha e_n$, where $\alpha \in \mathbb{Z}$ and $y$ is a root of
$A$. But then all the roots are also of the form $y + \alpha r$,
and since there are only finitely many choices for $y$, the claim
follows.

Now let $P$ be an indecomposable projective $H$-module. Then
$\tau^{-m}P$ is indecomposable for all $m \ge 0$, and the roots
$\{ [ \tau^{-m}P ] \}_{m=0}^{\infty}$ are pairwise distinct. By
the above, there exist integers $i < j$ and $v$ such that
$[\tau^{-i}P] + vr = [\tau^{-j}P]$. Denoting $j-i$ by $l$, we see
that
$$[\tau^{-(\alpha l +m)}P] = \alpha v r + [\tau^{-m}P]$$
for any $\alpha \ge 0$ and $m \in \{ 0, \dots, l-1 \}$. This shows
that the sequence $\{ \dim \tau^{-n}H \}_{n=1}^{\infty}$ grows
linearly.

(2) Suppose $H$ is wild, and let $P$ be an indecomposable
$H$-module. By \cite[Theorem 2.4]{Takane}, there exists an integer
$m$ such that
$$\lim_{n \to \infty} \frac{\dim \tau^{-n}P}{\rho^n n^{m-1}}$$
is nonzero, where $\rho$ is the spectral radius of the Coxeter
transformation of $H$. Now suppose that $\gamma \left (
\tau^{-1}_H \right )$ is finite, so that there exist a $t \ge 0$
and an $a \in \mathbb{R}$ such that $\dim \tau^{-n}P \le an^{t-1}$
for large $n$. Then
\begin{eqnarray*}
\lim_{n \to \infty} \frac{\dim \tau^{-n}P}{\rho^n n^{m-1}} & \le &
\lim_{n \to \infty} \frac{an^{t-1}}{\rho^n n^{m-1}} \\
& = & \lim_{n \to \infty} \frac{an^{t-m}}{\rho^n} \\
& = & 0
\end{eqnarray*}
since, by \cite[Theorem]{Ringel2}, the spectral radius $\rho$
satisfies $\rho > 1$. This is a contradiction, hence $\gamma \left
( \tau^{-1}_H \right ) = \infty$.
\end{proof}

We now prove the main result. It shows that the maximal complexity
in $\C_H$ (positive and negative) is either one, two or infinite,
depending on the representation type of $H$.

\begin{theorem}\label{mainclustercategory}
Let $H$ be a basic finite dimensional hereditary algebra over an
algebraically closed field, and let $\C_H$ be the corresponding
cluster category. Then
$$\sup \{ \cx^*_{\C_H}(X,Y) \mid X,Y \in \C_H \} = \left \{
\begin{array}{ll}
1 & \text{if $H$ has finite type,} \\
2 & \text{if $H$ has tame type,} \\
\infty & \text{if $H$ has wild type.}
\end{array}
\right.$$
\end{theorem}

\begin{proof}
Consider the subcategory $\add H$ of $\C_H$. It is contravariantly
finite since it contains only finitely many non-isomorphic
indecomposable objects. Moreover, by \cite[Theorem 3.3(b)]{BMRRT},
the following are equivalent for any object $X \in \T$:
\begin{enumerate}
\item $X \in \add H$, \item $\Hom_{\T}(H, \s X)=0$.
\end{enumerate}
Thus the object $H$ is cluster tilting in $\C_H$, and so from
Corollary \ref{maxcx} we see that
$$\sup \{ \cx^*_{\C_H}(X,Y) \mid X,Y \in \C_H \} =
\cx_{\C_H}^*(H,H).$$ Since $\C_H$ is Calabi-Yau, the maximal
positive complexity equals the maximal negative complexity. It
therefore suffices to prove the result for negative complexity.

By definition, the negative complexity $\cx_{\C_H}^-(H,H)$ equals
the rate of growth of the dimensions of the vector spaces
$\Hom_{\C_H}(H, \s^{-n} H)$ as $n$ grows. Since $\tau = \s$ on
$\C_H$, we obtain isomorphisms
\begin{eqnarray*}
\Hom_{\C_H}(H, \s^{-n} H) & \simeq & \Hom_{\C_H}(H, \tau^{-n} H)
\\
& \simeq & \bigoplus_{i \in \mathbb{Z}} \Hom_{\D^b(H)}( \tau^{-i}
\s^i H, \tau^{-n} H ) \\
& \simeq & \bigoplus_{i \in \mathbb{Z}} \Hom_{\D^b(H)}(H,
\tau^{i-n} \s^{-i} H)
\end{eqnarray*}
of vector spaces. If $H$ is of finite representation type, then
$\dim \Hom_{\C_H}(H, \tau^{-n} H)$ is bounded as $n \to \infty$.
Hence the result follows in this case.

Suppose $H$ is of infinite representation type. Given integers $i$
and $j$, the stalk complex $\tau^i \s^j H$ in $\D^b(H)$ is nonzero
in degree $j-1$ when $i \ge 1$, and in degree $j$ when $i \le 0$.
Thus when $n$ is positive, the only nonzero term in the above
direct sum appears when $i=0$, that is, the term $\Hom_{\D^b(H)}(H,
\tau^{-n} H)$. Therefore, for such $n$, we obtain the isomorphisms
\begin{eqnarray*}
\Hom_{\C_H}(H, \s^{-n} H) & \simeq & \Hom_{\D^b(H)}(H, \tau^{-n} H)
\\
& \simeq & \Hom_H(H, \tau^{-n}H) \\
& \simeq & \tau^{-n}H.
\end{eqnarray*}
Consequently, the negative complexity $\cx_{\C_H}^-(H,H)$ equals
the rate of growth of the sequence $\{ \dim \tau^{-n}H
\}_{n=1}^{\infty}$. The result now follows from Proposition
\ref{taugrowth}.
\end{proof}

\section{Cluster tilted algebras}

Let $H$ be a basic finite dimensional hereditary algebra over some
algebraically closed field, and $T$ a cluster tilting object in the
cluster category $\C_H$. The corresponding \emph{cluster tilted
algebra} is the endomorphism ring $\End_{\C_H}(T)$, itself a finite
dimensional algebra. By \cite{BMR}, the functor
\[ \Hom(T, -) \colon \mathcal{C}_H / (\tau T) \to \mod \left ( \End_{\C_H}(T) \right ) \]
is an equivalence, hence one might suspect from
Theorem~\ref{mainclustercategory} that tame cluster tilted algebras
have complexity two. However, this is \emph{not} the case: we show
in this section that their complexity is at most one.

In order to show this, we first recall some facts on Gorenstein
algebras. Let $\Gamma$ be such an algebra, and denote by $\CM (
\Gamma )$ the category of Cohen-Macaulay $\Gamma$-modules, i.e.\
$$\CM ( \Gamma ) = \{ M \in \mod \Gamma \mid \Ext_{\Gamma}^i(M, \Gamma )=0
\text{ for all } i >0 \}.$$ It follows from general cotilting
theory that this is a Frobenius exact category, in which the
projective injective objects are the projective $\Gamma$-modules,
and the injective envelopes are the left $\add
\Gamma$-approximations. Therefore the stable category
$\underline{\CM} ( \Gamma )$, which is obtained by factoring out
all morphisms which factor through projective $\Gamma$-modules, is
a triangulated category. Its shift functor is given by cokernels
of left $\add \Gamma$-approximations, the inverse shift is the
usual syzygy functor. Now let $\D^b ( \Gamma )$ be the bounded
derived category of finitely generated $\Gamma$-modules.
Furthermore, let $\D^{\text{perf}} ( \Gamma )$ be the thick
subcategory of $\D^b ( \Gamma )$ consisting of objects isomorphic
to bounded complexes of finitely generated projective
$\Gamma$-modules. It follows from work by Buchweitz, Happel and
Rickard (cf.\ \cite{Buchweitz}, \cite{Happel}, \cite{Rickard})
that $\underline{\CM} ( \Gamma )$ and the quotient category $\D^b
( \Gamma ) / \D^{\text{perf}} ( \Gamma )$ are equivalent as
triangulated categories.

The following lemma shows that if $\CM ( \Gamma )$ is of finite
type, i.e.\ contains only finitely many non-isomorphic
indecomposable objects, then the maximal complexity occurring in
$\D^b ( \Gamma )$ is either one or zero.

\begin{lemma}\label{gorenstein}
Let $\Gamma$ be a finite dimensional Gorenstein algebra such that
the category $\CM ( \Gamma )$ of Cohen-Macaulay $\Gamma$-modules
has finitely many non-isomorphic indecomposable objects. Then
\[ \sup \{ \cx^+_{\D^b(\Gamma)}(X,Y) \mid X,Y \in \D^b(\Gamma) \} = \left \{
\begin{array}{ll}
0 & \text{if $\Gamma$ has finite global dimension,} \\
1 & \text{otherwise}
\end{array}
\right. \]
\end{lemma}

\begin{proof}
Let $X$ and $Y$ be complexes in $\D^b ( \Gamma )$. As mentioned
above, the categories $\stabCM(\Gamma)$ and $\D^b ( \Gamma ) /
\D^{\text{perf}} ( \Gamma )$ are equivalent, and so there is a
dense functor $\D^b ( \Gamma ) \to \stabCM(\Gamma)$. If we denote
by $\overline{X}$ and $\overline{Y}$ the images of $X$ and $Y$ in
$\stabCM(\Gamma)$, then it follows from
\cite[Corollary~6.3.4]{Buchweitz} that $\cx^+_{\D^b(\Gamma)}(X,Y)
= \cx^+_{\stabCM(\Gamma)}(\overline{X},\overline{Y})$. Since
$\CM(\Gamma)$ is of finite type, we see that
$\cx^+_{\stabCM(\Gamma)}(\overline{X},\overline{Y})$ is at most
one.
\end{proof}

Recall from \cite{KellerReiten} that a cluster tilted algebra is
Gorenstein of dimension one. The following result shows that if
such an algebra $\La$ is tame, then $\CM ( \La )$ is of finite
type.

\begin{theorem}\label{cohenmacaulay}
If $\Lambda$ is a tame cluster tilted algebra, then the category
$\CM ( \La )$ of Cohen-Macaulay $\Lambda$-modules has finitely
many non-isomorphic indecomposable objects.
\end{theorem}

\begin{proof}
By definition, there is a hereditary algebra $H$ and a cluster
tilting object $T \in \mathcal{C}_H$ such that $\Lambda =
\End_{\mathcal{C}_H}(T)$. Moreover, by a theorem of Krause (cf.\
\cite[Corollary 3.4]{Krause}), the algebra $H$ is also tame.
Therefore, at least two of the indecomposable direct summands of $T$
lie in the non-regular component. Let $T_i$ by one such summand
lying ``as far to the right as possible'', that is, there is no path
in this component from $T_i$ to any other summands of $T$. We may
assume that $\tau^- T_i$ comes from a projective $H$-module. Now for
any $X \in \mod H$ we have
\begin{align*}
X \in \CM(\Lambda) & \iff \Ext^1_{\Lambda}(X, \Lambda) = 0 \\
& \iff \underline{\Hom}_{\Lambda}(\tau^- \Lambda, X) = 0 \\
& \implies \Hom_{\mathcal{C}_H} (\tau^- T_i, X) / (\text{maps factoring through } \tau T) = 0 \\
& \iff \Hom_{\mod H} (\tau^- T_i, X) / (\text{maps factoring through } \tau T) = 0
\end{align*}
For $X$ preprojective this is just $\Hom_{\mod H} (\tau^- T_i, X)$,
and this space only vanishes for finitely many $X$.

W denote by $R$ the direct sum of the regular summands of $T$. For
almost all regular and preinjective $H$-modules $X$ we have
\begin{align*}
& \Hom_{\mod H} (\tau^- T_i, X) / (\text{maps factoring through } \tau T) \\
& \qquad  = \Hom_{\mod H} (\tau^- T_i, X) / (\text{maps factoring through } \tau R).
\end{align*}
If $X$ lies in a homogeneous tube then the denominator vanishes,
and hence the space is non-zero.

For any indecomposable regular $X$ the dimension $\dim \Hom(\tau R,
X)$ is at most the number of indecomposable summands of $R$. For all
preinjective $X$ we have $\dim \Hom(\tau R, X) = \dim
\Hom(\tau^{1-\ell} R, X) = \dim \Hom(\tau R, \tau^{\ell} X)$, for
some $\ell$ only depending on $R$, and hence there is a common bound
for all the $\dim \Hom(\tau R, X)$ with $X$ preinjective. Now we
have
\begin{align*}
& \dim \Hom_{\mod H} (\tau^- T_i, X) / (\text{maps factoring through } \tau R) \\
\geqslant & \dim \Hom_{\mod H} (\tau^- T_i, X) - \dim \Hom_{\mod H}
(\tau^- T_i, \tau R) \cdot \underbrace{\dim \Hom_{\mod H} (\tau R,
X)}_{\text{bounded}}.
\end{align*}
Hence this space can only vanish if $\dim \Hom_{\mod H} (\tau^-
T_i, X)$ is sufficiently small. However this only happens for
finitely many modules which are preinjective or lie in
non-homogeneous regular tubes.
\end{proof}

Combining Lemma \ref{gorenstein} and Theorem \ref{cohenmacaulay},
we see that cluster tilted algebras of finite or tame type are of
complexity at most one.

\begin{theorem} \label{mainclustertilted}
If $\Lambda$ is a cluster tilted algebra of finite or tame
representation type, then
\[ \sup \{ \cx^+_{\D^b(\Lambda)}(X,Y) \mid X,Y \in \D^b(\Lambda) \} = \left \{
\begin{array}{ll}
0 & \text{if $\Lambda$ is hereditary,} \\
1 & \text{otherwise}
\end{array}
\right. \]
\end{theorem}

Next, we look at three examples of wild cluster-tilted algebras.
These examples show that Theorem~\ref{mainclustertilted}, and
hence also Theorem~\ref{cohenmacaulay}, does not generalize to
wild cluster tilted algebras. For background on mutations of
quivers with potentials, see \cite{BIRSm} and
\cite{DerksenWeymanZelevinsky}.

\begin{examples}
\begin{enumerate}
\item Let $\Lambda_0 = k[ \xymatrix{1 \ar@<0.5ex>[r]
\ar@<-0.5ex>[r]& 2 \ar[r] & 3}]$ be the path algebra of the wild
quiver $\xymatrix{1 \ar@<0.5ex>[r] \ar@<-0.5ex>[r] & 2 \ar[r] &
3}$. This is a hereditary algebra, and therefore
$\cx^*_{\D^b(\Lambda_0)}(X, Y) = 0$ for any $X, Y \in
\D^b(\Lambda_0)$. \item Let $\Lambda_1$ be the cluster tilted
algebra obtained from $\Lambda_0$ by mutation at the vertex $2$.
Then $\Lambda_1$ is the path algebra of the quiver
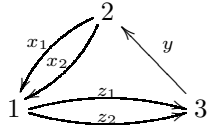
\[ \xymatrix{
& 2 \ar@<-0.4ex>@/^/[dl]_*-{\labelstyle x_2} \ar@/_/[dl]_*-{\labelstyle x_1} \\
1 \ar@<-0.5ex>@/^/[rr]^*-{\labelstyle z_1}
\ar@/_/[rr]^*-{\labelstyle z_2} && 3 \ar[ul]_{y} }
\] subject to the relations given by the cyclic derivatives of the
potential $x_1 z_1 y + x_2 z_2 y$, namely the relations $\{ x_1
z_1 + x_2 z_2, y x_1, y x_2, z_1 y, z_2 y \}$. Then
\[ \sup \{ \cx^+_{\D^b(\Lambda_1)}(X,Y) \mid X,Y \in \D^b(\Lambda_1) \} = 1. \]
\item Let $\Lambda_2$ be the cluster tilted algebra obtained from
$\Lambda_1$ by mutation at the vertex $1$. Then $\Lambda_2$ is the
path algebra of the quiver
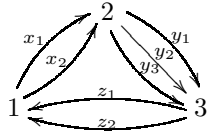
\[ \xymatrix{
& 2 \ar@<0.6ex>@/^/[dr]^*-{\labelstyle y_1}
\ar@<-0.4ex>@/_/[dr]^*-{\labelstyle y_3} \ar[dr]^*-{\labelstyle y_2} \\
1 \ar@<0.4ex>@/^/[ur]^*-{\labelstyle x_1}
\ar@<0.5ex>@/_/[ur]^*-{\labelstyle x_2} && 3
\ar@<0.5ex>@/^/[ll]_*-{\labelstyle z_2}
\ar@<0.6ex>@/_/[ll]_*-{\labelstyle z_1} }
\] subject to the relations given by the cyclic derivatives of the
potential $x_1 y_1 z_2 + x_2 y_2 z_1 + x_1 y_3 z_1 - x_2 y_3 z_2$,
namely the relations $\{ x_1 y_3 + x_2 y_2, x_1 y_1 - x_2 y_3, y_1
z_2 + y_3 z_1, y_2 z_1 - y_3 z_2, z_1 x_1, z_2 x_2, z_1 x_1 - z_2
x_2 \}$. Then
\[ \sup \{ \cx^+_{\D^b(\Lambda_2)}(X,Y) \mid X,Y \in \D^b(\Lambda_2) \} = 2. \]
\end{enumerate}
\end{examples}

\begin{proof}
The claims for $\La_0$ are clear. For $\Lambda_1$, the
indecomposable projectives have the following composition
structures:
\[ \xymatrix{
& 1 \ar[ld]_{x_1} \ar[rd]^{x_2} &&& 2 \ar[d]^{y} &&& 3 \ar[ld]_{z_1} \ar[rd]^{z_2} \\
2 && 2 && 3 && 1 \ar[ld]_{x_2} \ar[rd]_{x_1} && 1 \ar[ld]^{-x_2} \ar[rd]^{x_1} \\
&&&&& 2 && 2 && 2 } \] We see that the three simple modules
satisfy
\[ \Omega_{\La_1}^1(S_1) = S_2^2 \qquad \qquad \Omega_{\La_1}^1(S_2) = S_3 \qquad \qquad
\Omega_{\La_1}^2(S_3) = S_2. \] Consequently every simple
$\Lambda_1$-module is eventually $\Omega$-periodic, and therefore
\[ \sup \{ \cx^+_{\D^b(\Lambda_1)}(X,Y) \mid X,Y \in \D^b(\Lambda_1) \} = 1. \]
For $\Lambda_2$, the indecomposable projective modules have the
following composition structures:
\[ \xymatrix{
&&&& 1 \ar[ld]_{z_1} \ar[rd]^{z_2} \\
&&& 3 \ar[lld]_{y_1} \ar[d]_{y_3} \ar[rrd]^<<<<<<{y_2} && 3 \ar[lld]_<<<<<<{-y_1} \ar[d]^{y_3} \ar[rrd]^{y_2} \\
& 2 \ar[ld]_{x_2} \ar[rd]^{x_1} && 2 \ar[ld]_{x_2} \ar[rd]^{x_1} && 2 \ar[ld]_{-x_2} \ar[rd]^{x_1}
&& 2 \ar[ld]_{-x_2} \ar[rd]^{x_1} \\
1 && 1 && 1 && 1 && 1 } \] $$\xymatrix{
& 2 \ar[ld]_{x_1} \ar[rd]^{x_2} \\
1 \ar[rd]_{z_1} && 1 \ar[ld]^{z_2} \\
& 3 } $$ \[ \xymatrix{
&&& 3 \ar[lld]_{y_1} \ar[d]_{y_3} \ar[rrd]^{y_2} \\
& 2 \ar[ld]_{x_2} \ar[rd]^{x_1} && 2 \ar[ld]_{x_2} \ar[rd]^{x_1} && 2 \ar[ld]_{-x_2} \ar[rd]^{x_1} \\
1 && 1 && 1 && 1
} \]
If we denote by $M_n$ the module with composition structure
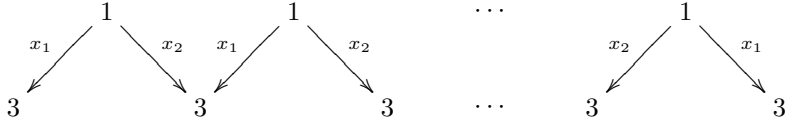
\[ \xymatrix{
& 1 \ar[ld]_{x_1} \ar[rd]^{x_2} && 1 \ar[ld]_{x_1} \ar[rd]^{x_2} && \cdots && 1 \ar[ld]_{x_2}
\ar[rd]^{x_1} \\
3 && 3 && 3 & \cdots & 3 && 3 } \] (with $n$ composition factors
$S_1$ and $n+1$ composition factors $S_3$), then one can show by
direct calculation that
\[ \Omega_{\La_2}^2(M_n) = M_{n+2}. \]
Now note that
\[ S_3 = M_0 \qquad \text{and} \qquad \Omega_{\La_2}^3(S_1) = M_1, \]
so the rate of growth of the dimensions of the syzygies of these
simple modules is linear. Finally, note that
$\Omega_{\La_2}^1(S_2)$ is an extension of $S_1 \oplus S_1$ and
$S_3$, hence the rate of growth of the dimensions of its syzygies
is at most linear. This shows that
\[ \sup \{ \cx^+_{\D^b(\Lambda_2)}(X,Y) \mid X,Y \in \D^b(\Lambda_2) \} = 2. \qedhere \]
\end{proof}

\begin{remark}
As mentioned, the above example not only shows that
Theorem~\ref{mainclustertilted} does not generalize to wild
cluster tilted algebras. It also shows that the same is true for
Theorem~\ref{cohenmacaulay}, that is, there exist wild cluster
tilted algebras with infinitely many non-isomorphic indecomposable
Cohen-Macaulay modules. Namely, by Lemma \ref{gorenstein}, the
algebra in example (3) has this property.
\end{remark}

We conclude this paper with the following more general questions
on the complexity of wild cluster tilted algebras:

\noindent
\begin{questions}
\begin{enumerate}
\item What numbers occur as
\[ \sup \{ \cx^+_{\D^b(\Lambda)}(X,Y) \mid X,Y \in \D^b(\Lambda) \} \]
for $\Lambda$ a cluster tilted algebra of wild type? \item Given a
wild hereditary algebra $H$, do all the numbers in (1) occur as
\[ \sup \{ \cx^+_{\D^b(\Lambda)}(X,Y) \mid X,Y \in \D^b(\Lambda) \} \]
for some cluster tilted algebra $\Lambda$ of type $H$?
\end{enumerate}
\end{questions}

\end{document}